\pdfoutput=1
\documentclass[12pt,reqno]{article}

\usepackage[T1]{fontenc}
\usepackage[utf8]{inputenc}
\usepackage{textcomp}

\usepackage[usenames]{color}
\definecolor{webgreen}{rgb}{0,.5,0}
\definecolor{webbrown}{rgb}{.6,0,0}

\usepackage[colorlinks=true,
  linkcolor=webgreen,
  filecolor=webbrown,
  citecolor=webgreen]{hyperref}

\usepackage{graphicx}
\usepackage{amscd}
\usepackage{amsmath,amssymb,amsthm}
\usepackage{amsfonts,latexsym}

\font\smallit=cmti10
\font\smalltt=cmtt10
\font\smallrm=cmr9

\newtheorem{theorem}{Theorem}[subsection]

\newtheorem{definition}[theorem]{Definition}

\newtheorem{lemma}[theorem]{Lemma}

\begin{document}

\begin{center}
{\bf ON ASYMPTOTIC FORMULA OF THE PARTITION FUNCTION $p_A(n)$. }
\vskip 20pt
{\bf A. David Christopher}\\
{\smallit Department of Mathematics, The American College, Madurai, India}\\
\href{mailto:davchrame@yahoo.co.in}{davchrame@yahoo.co.in}\\
{\bf M. Davamani Christober}\\
{\smallit Department of Mathematics, The American College, Madurai, India}\\
\href{mailto:jothichristopher@yahoo.com}{jothichristopher@yahoo.com}
\end{center}

\vskip 30pt
\centerline{\smallit Received:, Accepted:, Published:}
\vskip 30pt
\centerline{\bf Abstract}
\noindent
The partition function, $p_A(n)$, is defined to be the number of partitions of $n$ with parts in the set A, where $n$ is a positive integer and $A$ is a set of positive integers. It is well documented that: if A is a finite set with $\gcd(A)=1$ and $|A|=k$, then 
\[p_A(n)\sim \frac{n^{k-1}}{(\prod_{a\in A}a)(k-1)!}.
\]
Number of proofs have been obtained for this estimate. In this article, we give a new proof for the above estimate by making use of the fact that: $p_A(n)$ is a $quasi\  polynomial$ when A is a finite set. Present method of proof is purely combinatorial.

\pagestyle{myheadings}
\markright{\smalltt INTEGERS:\smallrm ELECTRONIC JOURNAL OF COMBINATORIAL NUMBER THEORY}
\thispagestyle{empty}
\baselineskip=15pt

\section{Introduction and motivation.}
 By a partition of a positive integer $n$, we mean a non increasing sequence of positive integers say $\pi=\left(x_1,x_2,\cdots,x_m\right)$ such that $\sum_{i=1}^{m}x_{i}=n$. Each $x_i$ is called a part of the partition $\pi$. Let $A$ be a set of positive integers. The partition function, $p_A(n)$, is defined to be the number of partitions of $n$ with parts in the set A. \par 
The generating function of $p_A(n)$ is 
\begin{equation}\label{eqn1}
\sum_{n=0}^{\infty}p_A(n)x^n=\prod_{a\in A}\frac{1}{1-x^a}
\end{equation}
with $p_A(0)=1$; this generating function is valid in the interval $|x|<1$.\par
We notice that, if $\gcd(A)\neq 1$, then 
\[p_A(n)=\left\{\begin{array}{l}
p_{\frac{A}{\gcd(A)}}\left(\frac{n}{\gcd(A)}\right) \ \ if\ \gcd(A)|n, \\
0\ otherwise
\end{array}\right.
\]
where the set $\frac{A}{\gcd(A)}=\left\{\frac{a}{\gcd(A)}:a\in A\right\}$.\par 
Throughout this paper, we let $A=\left\{a_1,a_2,\cdots, a_k\right\}$ where $a_i'$s are positive integers; unless otherwise stated, we proceed with the understanding that $\gcd(a_1,a_2,\cdots,a_k)=1$.\par 
There is a lot of literature available on the function $p_A(n)$. T. C. Brown et al \cite{brown} found exact formulas for $p_A(n)$  when $|A|=2$ or $3$. 
Exact formula for $p_A(n)$ can be found using partial fraction decomposition of its generating function (see \cite{natb}). Gert Almkvist \cite{Gert} provided an exact formula for $p_A(n)$ without the use of partial fraction decomposition of generating function.  Following estimate of $p_A(n)$ is well known:
\begin{equation}\label{eqn2}
p_A(n)\sim \frac{n^{k-1}}{\left(a_1a_2\cdots a_k\right)(k-1)!}.
\end{equation}
This estimate was proved by several authors. Earlier, a proof was given by E. Netto \cite{Netto} and later G. Polya and G. Szeg\"{o} \cite{polya} gave another proof; they have used partial fraction decomposition of generating function to accomplish the purpose. Paul Erdos \cite{erd} proved the estimate when $A=\left\{1,2,\cdots,k\right\}$. S. Sertoz and A. E. Ozl\"{u}k. \cite{ser} had given a proof based on the following reccurrence relation of $p_A(n)$:
\[1=\sum_{i=n-k+2}^{n}C_{n-i}[p_A(i)-p_A(i-(a_1\cdots a_k))]
\]
for $n>(a_1\cdots a_k)-(a_1+\cdots +a_k)+k-2$, where
\[C_m=\left\{\begin{array}{l}
(-1)^m{{k-2} \choose {m}} \ \ for\ 0\leq m\leq k-2 \\
0\ otherwise
\end{array}\right.
\]An arithmetic proof was obtained by Melvyn B. Nathanson \cite{natj}.\par 
 In this note, we give a new proof for the above mentioned estimate of $p_A(n)$. Present proof is based on the quasi polynomial representation of $p_A(n)$. 
\begin{definition}
An arithmetical function $f$ is said to be a $Quasi\  polynomial$ if, $f(\alpha l+r)$ is a polynomial in $l$ for each $r=0,1,\cdots,\alpha-1$, where $\alpha$ is a positive integer greater than 1. Each polynomial $f(\alpha l+r)$ is called constituent polynomial of $f$ and $\alpha$ is called quasi period of $f$. 
\end{definition}
The fact that, $p_A(n)$ is a quasi polynomial, was proved by  E. T. Bell \cite{bell}, who empolyed partial fraction decomposition of its generating function. E. M. Wright \cite{wright} showed that $p_A(n)$ is a quasi polynomial by extracting the term $(1-x^t)^{-k}$ from the generating function of $p_A(n)$, where $t=lcm(a_1,\cdots,a_k)$. Using similar method, O. J. R\o dseth and J. A. Sellers \cite{stein} also had found the quasi polynomial representation of $p_A(n)$, but in binomial coefficients form.\par  
In the present proof, we derive a finite linear recurrence relation of $p_A(n)$.
Using that recurrence relation, we show that: the function, $p_A(n)$, is a quasi polynomial with quasi period $a_1a_2\cdots a_k$ and each constituent polynomial being of degree $k-1$. Further, we found that, the leading coefficient of each constituent polynomial of $p_A(n)$ is $\frac{(a_1a_2\cdots a_k)^{k-2}}{(k-1)!}$. Accordingly, we get the required estimate.

\section{Proof.}

\subsection{A recurrence relation satisfied by $p_A(n)$ }
Following lemma is crucial to our proof.
\begin{lemma}\label{lemma}Let $n$ be a positive integer. For $\forall a\in A$, we have
\begin{equation}p_A(n)=p_A(n-a)+p_{A\setminus\left\{a\right\}}(n)
\end{equation}
provided $a\leq n$.
\end{lemma}
\begin{proof}
Let $\pi=\left(x_1,x_2,\cdots,x_m\right)$ be a partition of $n$ with parts in the set A and let $a\in A$.\par 
\it{Case (i)}\normalfont: 
Assume that $x_i=a$ for some $i$. In such case, we can write\\ $\pi=\left(x_1,\cdots,x_{i-1},a,x_{i+1},\cdots,x_m\right)$. We notice that, the mapping
\[\left(x_1,\cdots,x_{i-1},a,x_{i+1},\cdots,x_m\right)\to \left(x_1,\cdots,x_{i-1},x_{i+1},\cdots,x_m\right)
\] 
is a bijection between the following sets:
\begin{itemize}
\item Set of all partitions of $n$ with parts in the set A and having the part $a$.
\item Set of all partitions of $n-a$ with parts in the set A.
\end{itemize}   We see that, cardinality of the latter set is $p_A(n-a)$.\par 
\it{Case(ii)}: \normalfont Assume that $x_i\neq a$ $\forall i=1,2,\cdots,m$. It is not hard to see that, enumeration of such partitions is $p_{A\setminus\left\{a\right\}}(n)$. Thus result follows.
\end{proof}

\subsection{Main part of the proof}
As consequences of lemma \ref{lemma}, we show that:
\begin{enumerate}
\item The function $p_A(n)$ is a quasi polynomial with quasi period $a_1a_2\cdots a_k$ and each constituent polynomial is of degree $k-1$. 
\item Leading coefficient of the constituent polynomial $p_A(a_1a_2\cdots a_kl+r)$ is $\frac{(a_1a_2\cdots a_k)^{k-2}}{(k-1)!}$ for each $r=0,1,\cdots,a_1\cdots a_k-1$.
\end{enumerate} \par
 From the above two statements it follows readily that:
\[\lim_{l\to\infty}\frac{p_A(a_1a_2\cdots a_kl+r)}{(a_1a_2\cdots a_kl+r)^{k-1}}=\frac{1}{\left(a_1a_2\cdots a_k\right)(k-1)!}
\]
for each $r=0,1,\cdots ,a_1a_2\cdots a_k-1$. Whence the targeted estimate follows. 
\par
Now, we prove statement 1 and statement 2 simultaneously by using induction on $k$.
\par 
Suppose that $k=2$. Let $A=\left\{a_1,a_2\right\}$ with $\gcd(a_1,a_2)=1$. Then applying lemma \ref{lemma} $a_1$ times, we get
\begin{equation}\label{eqn5}
p_A(a_1a_2l+r)-p_A(a_1a_2(l-1)+r)=\sum_{i=0}^{a_1-1}p_{\left\{a_1\right\}}(a_1a_2l+r-ia_2)
\end{equation}
for each $r=0,1,\cdots,a_1a_2-1$. Since the congruence equation
\[a_2x\equiv r(mod\ a_1)
\]
has unique solution modulo $a_1$ (see \cite{david} pp. 83-84), the right side of the equation (\ref{eqn5}) equals $1$. Then replacing $l$ by $1,\cdots,l$ in equation (\ref{eqn5}) and adding, we get that
\[p_A(a_1a_2l+r)=l+p_A(r)
\]
for each $r=0,1,2,\cdots,a_1a_2-1$. Thus, the function $p_A(n)$ is a quasi polynomial with each constituent polynomial being of degree 1 and leading coefficient $1=\frac{(a_1a_2)^{2-2}}{(2-1)!}$.\par

Assume that, the result is true when $|A|<k$ for a fixed $k\geq 3$. Now, consider the set of $k$ positive integers say $A=\left\{a_1,a_2,\cdots,a_k\right\}$ with $\gcd(a_1,a_2,\cdots,a_k)=1$.  
Let $s=\gcd(a_1,a_2,\cdots ,a_{k-1})$. Then applying lemma \ref{lemma} $a_1a_2\cdots a_{k-1}$ times again, we get

\begin{equation}\label{eqn4}
p_A(a_1\cdots a_kl+r)-p_A(a_1\cdots a_k(l-1)+r)=\sum_{0\leq i\leq a_1\cdots  a_{k-1}-1; s|(r-ia_k)}p_{\left\{a_1,\cdots,a_{k-1}\right\}}(a_1\cdots a_{k}l+r-ia_k)
\end{equation}
\[= \sum_{0\leq i\leq a_1\cdots  a_{k-1}-1; s|(r-ia_k)}p_{\left\{\frac{a_1}{s},\cdots ,\frac{a_{k-1}}{s}\right\}}\left(\frac{a_1}{s}\cdots \frac{a_{k-1}}{s}(a_ks^{k-2}l+q_i)+r_i\right)
\]
for each $r=0,1,\cdots ,a_1a_2\cdots a_k-1$, where $r_i$ and $q_i$ were determined from the equality: $\frac{r-ia_k}{s}=\frac{a_1\cdots a_{k-1}}{s^{k-1}}q_i+r_i$; uniqueness of $r_i$ and $q_i$ and the bound: $0\leq r_i\leq \frac{a_1\cdots a_{k-1}}{s^{k-1}}-1$ follows as a consequence of Division algorithm.\par 
We notice that, the congruence equation
\begin{equation}\label{eqn3}
a_kx\equiv r(mod\ s)
\end{equation}
has solution if, and only if, $\gcd(a_k,s)|r$ (see \cite{david} pp. 83-84). In such case, eqn(\ref{eqn3}) will have $\gcd(a_k,s)$ number of mutually incongruent solution modulo $s$. Here, $\gcd(a_k,s)=1$; therefore, eqn(\ref{eqn3}) has a unique solution modulo $s$. \par 
Since $\gcd(\frac{a_1}{s},\cdots \frac{a_k}{s})$=1, by induction assumption, it follows that the right side of the equation (\ref{eqn4}) is the sum of $\frac{a_1\cdots a_k}{s}$ polynomials, each of which is of degree $k-2$ and with leading coefficient being $\left(\frac{a_1\cdots a_{k-1}}{s^{k-1}}\right)^{k-3} \frac{a_k^{k-2}s^{(k-2)^2}}{(k-2)!}$. Consequently, the right side sum of equation (\ref{eqn4}) is itself a polynomial of degree $k-2$. This implies that $p_A(a_1\cdots a_kl+r)$ is a polynomial in $l$ of degree $k-1$ for each $r=0,1,\cdots,a_1\cdots a_k-1$.\par 
Now, we calculate the leading coefficient of $p_A(a_1\cdots a_kl+r)$.  If one denotes the leading coefficient of the polynomial $p_A(a_1\cdots a_kl+r)$ by $c_{k-1}$, then by the previous observations it follows that
\[(k-1)c_{k-1}=\frac{(a_1\cdots a_k)^{k-2}}{(k-2)!},
\]
which simplifies to
\[c_{k-1}=\frac{(a_1\cdots a_k)^{k-2}}{(k-1)!}.
\]
The proof is now completed.


\begin{thebibliography}{10}
\bibitem{Gert} Gert Almkvist, {\em Partitions with parts in a finite set and with parts outside a finite set}, Experimental Mathematics, Vol. 11(2002), No. 4, pp 449-456.
\bibitem{bell}
E. T. Bell, {\em Interpolated denumerants and Lambert series}, Amer. J. Math. 65 (1943), pp. 382-386.
\bibitem{brown} T. C. Brown, Wun- Seng Chou, Peter J.-S. Shiue, {\em On the partition function of a finite set }, Australasian Journal of Combinatorics 27(2003), pp. 193-204.
\bibitem{david} David M. Burton, {\em Elementary Number theory}, Allyn and Bacon, Inc, 1980.
\bibitem{erd} P. Erd\"{o}s, {\em On an elementary proof of some asymptotic formulas in the theory of partitions}, Ann. of Math. (2) 43(1942), 437-450.

\bibitem{natb} Melvyn B. Nathanson, {\em Elementary methods in Number theory}, Springer-verlag, New York- Berlin- Heidelberg (2000), pp. 466-467.
\bibitem{natj} Melvyn B. Nathanson, {\em Partition with parts in a finite set}, Proc. Amer. Math. Soc. 128(2000), 1269-1273.
\bibitem{Netto} E. Netto, {\em Le hrbuch der Combinatorik}, Teubner, Leipzig, 1927.
\bibitem{polya} G. Polya and G. Szeg\"{o}, {\em Aufgaben and Lehrs\"{a}tze aus der analysis}, Springer- Verlag, Berlin, 1925. English translation: {\em Problems and Theorems in Analysis}, Springer- verlag, New York, 1972.
\bibitem{stein}
\O{}.J. R\o dseth and J.A. Sellers, {\em Partition with parts in a finite set}, Int. J. Number Theory 02, 455(2006). 
\bibitem{ser} S. Sertoz and A.E. Ozl\"{u}k, {\em On the number of representations of a integer by a linear form}, Istabbul Univ. FenFak. Mat. Derg 50(1991), 67-77.


\bibitem{wright} E. M. Wright, {\em A simple proof of a known result in partitions}, Amer. Math. Monthly 68(1961), 144-145.
\end{thebibliography}
\end{document}